\documentclass[12pt]{amsart}
\usepackage{latexsym,amssymb,amscd}

\def\NZQ{\mathbb}               

\def\ZZ{{\NZQ Z}}
\def\RR{{\NZQ R}}

\def\B'c{{\mathcal{B'}}}
\def\U'c{{\mathcal{U'}}}

\def\opn#1#2{\def#1{\operatorname{#2}}} 
\opn\chara{char}
\opn\length{\ell}
\opn\projdim{proj\,dim}
\opn\injdim{inj\,dim}
\opn\ini{in}
\opn\rank{rank}
\opn\depth{depth}
\opn\sdepth{sdepth}
\opn\height{ht}
\opn\embdim{emb\,dim}
\opn\codim{codim}

\opn\Tr{Tr}
\opn\bigrank{big\,rank}
\opn\superheight{superheight}\opn\lcm{lcm}
\opn\trdeg{tr\,deg}%
\opn\reg{reg}
\opn\lreg{lreg}
\opn\set{set}
\opn\supp{Supp}
\opn\shad{Shad}
\opn\div{div}
\opn\Div{Div}
\opn\cl{cl}
\opn\Cl{Cl}
\opn\Spec{Spec}
\opn\Supp{Supp}
\opn\supp{supp}
\opn\Sing{Sing}
\opn\Ass{Ass}
\opn\Ann{Ann}
\opn\Rad{Rad}
\opn\Soc{Soc}
\opn\Ker{Ker}
\opn\Coker{Coker}
\opn\Im{Im}
\opn\Hom{Hom}
\opn\Tor{Tor}
\opn\Ext{Ext}
\opn\End{End}
\opn\Aut{Aut}
\opn\id{id}

\opn\nat{nat}
\opn\GL{GL}
\opn\SL{SL}
\opn\mod{mod}
\opn\ord{ord}
\opn\aff{aff}
\opn\con{conv}
\opn\relint{relint}
\opn\st{st}
\opn\lk{lk}
\opn\cn{cn}
\opn\core{core}
\opn\vol{vol}
\opn\gr{gr}

\def\pot#1#2{#1[\kern-0.28ex[#2]\kern-0.28ex]}

\opn\dirlim{\underrightarrow{\lim}}
\opn\invlim{\underleftarrow{\lim}}
\def\pnt{{\raise0.5mm\hbox{\large\bf.}}}

\def\Implies{\ifmmode\Longrightarrow \else
     \unskip${}\Longrightarrow{}$\ignorespaces\fi}
\def\implies{\ifmmode\Rightarrow \else
     \unskip${}\Rightarrow{}$\ignorespaces\fi}
\def\iff{\ifmmode\Longleftrightarrow \else
     \unskip${}\Longleftrightarrow{}$\ignorespaces\fi}

\let\:=\colon
\newtheorem{Theorem}{Theorem}[section]
\newtheorem{Lemma}[Theorem]{Lemma}
\newtheorem{Corollary}[Theorem]{Corollary}
\newtheorem{Proposition}[Theorem]{Proposition}

\newtheorem{Example}[Theorem]{Example}

\newcommand{\beq}[1]{\begin{equation}\label{#1}}
\newcommand{\eeq}{\end{equation}}

\newtheorem{theorem}{Theorem}[section]
\newcommand{\bet}[1]{\begin{theorem}\label{#1}}
\newcommand{\eet}{\end{theorem}}
\newtheorem{lemma}[theorem]{Lemma}
\newcommand{\bel}[1]{\begin{lemma}\label{#1}}
\newcommand{\eel}{\end{lemma}}
\newcommand{\bep}{\begin{proof}}
\newcommand{\eep}{\end{proof}}
\let\epsilon=\varepsilon
\let\phi=\varphi
\let\kappa=\varkappa
\numberwithin{equation}{section}

\title{On the behaviour of Stanley depth under variable adjunction}

\author[Mihai Cipu]{Mihai Cipu}
\address{Institute of Mathematics of the Romanian Academy,
P.O. Box 1-764, RO-014700 Bucharest, Romania}
\email{Mihai.Cipu@imar.ro}

\author[Muhammad Imran Qureshi]{Muhammad Imran Qureshi}
\address{Abdus Salam School of Mathematical Sciences, GC
University, Lahore, 68-B New Muslim town Lahore, Pakistan}
\email{imranqureshi18@gmail.com}
\thanks{Both authors are grateful to Professor D. Popescu
for  helpful discussions
during the preparation of this paper.}

\begin{document}

\maketitle

\begin{abstract}
Let $S=K[x_1,\ldots,x_n]$ be a polynomial ring in $n$ variables
over the field $K$. For integers $1\leq t< n$ consider the ideal $I=(x_1,\ldots,x_t)\cap(x_{t+1}, \ldots,x_n)$ in $S$. In this
paper we  bound from above the Stanley depth of the ideal $I'=(I,x_{n+1},\ldots,x_{n+p})\subset S'=S[x_{n+1},\ldots,x_{n+p}]$.
We give similar upper bounds for the Stanley depth of the ideal $(I_{n,2},x_{n+1},\ldots,x_{n+p})$, where $I_{n,2}$ is the
squarefree Veronese ideal of degree 2 in $n$ variables.
\end{abstract}

\section{Introduction}
Let $S=K[x_1,\ldots,x_n]$ be the polynomial ring  in $n$ variables
over a field  $K$ and $M$ be a finitely generated $\ZZ^n$-graded
$S$-module. If $u\in M$ is a homogeneous element in $M$ and
$Z\subset \{x_1,\ldots,x_n\}$ then let $uK[Z]\subset M$ denote
the linear $K$-subspace of all elements of the form $uf$,
$f\in K[Z]$. This space is called a Stanley space of dimension
$|Z|$ if $uK[Z]$ is a free $K[Z]$-module. A Stanley decomposition
of module $M$ is a presentation of the $K$-vector space $M$ as
a finite direct sum of Stanley spaces
$\mathcal{D}:M=\bigoplus^r _{i=1} u_i K[Z_i]$. Set
$\sdepth(\mathcal{D})=\min\{|Z_i|: i=1,\ldots,r\}$. The number
\[\sdepth(M):=\max\{\sdepth(\mathcal{D}):\mathcal{D} \text{ is a
 Stanley decomposition of}\ M\}
\]
is called the Stanley depth of $M$. Stanley depth is an invariant
which has some common properties with the homological depth
invariant.

In 1982 Stanley conjectured (see \cite{RPS}) that
$\sdepth M\geq\depth M$.
This conjecture is still open except some results obtained mainly
for $n\leq 5$ (see~\cite{A}, \cite{AP}, \cite{DP}, \cite{PQ},
\cite{R}). A method to compute the Stanley depth is given in
\cite{HVZ}. Even when it does not provide the value of the
Stanley depth, this method allows one to obtain fairly good
estimations for the invariant of interest.

In~\cite{I}, Ishaq proved that if $J$ is a monomial ideal of
$S =K[x_1,\ldots,x_n]$ and $J'=(J,x_{n+1})$ is the ideal of
$S'=S[x_{n+1}]$ then $\sdepth(J)\leq \sdepth(J')\leq
\sdepth(J)+1$. When adjoining more variables, a similar
result can be easily obtained by iterating Ishak's result.
However, the upper bound for $\sdepth((J,x_{n+1}, \ldots,
x_{n+p}))$ ($p\ge 2$) thus obtained can
be sometimes too pessimistic.

The aim of this paper is to bound from above the Stanley depth of
ideals obtained by adjoining variables to monomial  ideals in $S$
belonging to two classes. A first class consists of radical
monomial ideals described in the theorem below, whose proof is
given in the next section.

\medskip

\textbf{Theorem~\ref{te1}.} \textit{
Let $I=(x_1,\ldots,x_t)\bigcap(x_{t+1},\ldots,x_n)$ be a
monomial ideal in \linebreak
$S =K[x_1,\ldots,x_n]$, where $1\leq t< n$,
and let $I'=(I,x_{n+1},\ldots,x_{n+p}) \subset S'=S[x_{n+1},
\ldots, x_{n+p}]$, where $p\geq 2$.
 Then}
\beq{ect1}
\sdepth(I')\leq 2+\frac{\left(\!\!
                          \begin{array}{c}
                            n \\
                            3 \\
                          \end{array}
                        \!\!\right)
                        -
                        \left(\!\!
                          \begin{array}{c}
                            t \\
                            3 \\
                          \end{array}
                        \!\!\right)
                        -
                        \left(\!\!
                          \begin{array}{c}
                            n-t \\
                            3 \\
                          \end{array}
                        \!\!\right)
                        +p\left(\!\!
                           \begin{array}{c}
                             n \\
                             2 \\
                           \end{array}
                         \!\!\right)+n\left(\!\!
                                   \begin{array}{c}
                                     p \\
                                     2 \\
                                   \end{array}
                                \!\! \right)+\left(\!\!
                                           \begin{array}{c}
                                             p \\
                                             3 \\
                                           \end{array}
                                        \!\! \right)
}{t(n-t)+np-\frac{p(n+2)}{4}
} .
\eeq
\medskip

An alternative bound is obtained by imposing some conditions on
$t$, $n$ and $p$, see Theorem ~\ref{te2} in Section~\ref{sec2}
for the precise statement.

The reasoning used to prove the results mentioned above can be
adapted to work for another class of ideals, namely, squarefree
Veronese ideals of degree 2. In Section~\ref{sec3} we shall prove
the following.

\medskip

\textbf{Theorem~\ref{te3}.} \textit{
Let $I_{n,2}$ be the squarefree Veronese ideal of degree 2
in $S$ and $(I_{n,2},x_{n+1},\ldots,x_{n+p})$ be the squarefree
ideal in $S'$, where $p\ge 2$. Then}
\[
\sdepth(I_{n,2},x_{n+1},\ldots,x_{n+p})\leq 2+\frac{\left(\!\!
                            \begin{array}{c}
                              n+p \\
                              3 \\
                            \end{array}
                          \!\!\right)
}{\left(\!\!
    \begin{array}{c}
      n \\
      2 \\
    \end{array}
 \!\! \right)+np-p- \frac{p}{2} \lfloor(
                                     \begin{array}{c}
                                       n \\
                                       3 \\
                                     \end{array}
                                   )/(
                                     \begin{array}{c}
                                       n \\
                                       2 \\
                                     \end{array}
                                   )
  \rfloor
}.
\]

\medskip

Also this bound is further improved by
imposing some condition on $n$ and $p$ (cf. Theorem~\ref{te4}).

Herzog, Vl\u adoiu and Zhang~\cite{HVZ} have results implying
that Stanley's conjecture is true for squarefree Veronese
ideals. In Section~\ref{sec3} we note that Stanley's conjecture
is valid for the ideal obtained by adding several variables
to a squarefree Veronese ideal.

\medskip

\textbf{Proposition~\ref{pr1}.} \textit{
Let $I\subset S=K[x_1,\ldots,x_n]$ be the squarefree
Veronese ideal generated by all monomials of degree $d$ and
$I'=(I,x_{n+1},\ldots,x_{n+m})\subset S'=
S[x_{n+1},\ldots,x_{n+m}]$. Then
Stanley's conjecture holds for the  ideal $I'$.
}

\medskip

In the last section of the paper we compare the bounds which
we obtained without conditions with those which we obtained
when appropriate conditions are imposed.

Some results from~\cite{MC}, \cite{S}, \cite{I}
and~\cite{S1} are very important for our estimations
of Stanley depth and precise references will be given in
appropriate places.  For unexplained notation, the reader
is referred  to~\cite{HVZ}.

\section{Upper bounds for the Stanley depth of squarefree
monomial ideal when some variables are added}  \label{sec2}

\begin{Theorem} \label{te1}
Let $I=(x_1,\ldots,x_t)\bigcap(x_{t+1},\ldots,x_n)$ be a
monomial ideal in \linebreak
$S =K[x_1,\ldots,x_n]$, where $1\leq t< n$,
and let $I'=(I,x_{n+1},\ldots,x_{n+p}) \subset S'=S[x_{n+1},
\ldots, x_{n+p}]$, where $p\geq 2$.
Then
\[\sdepth(I')\leq 2+\frac{\left(
                          \begin{array}{c}
                            n \\
                            3 \\
                          \end{array}
                        \right)
                        -
                        \left(
                          \begin{array}{c}
                            t \\
                            3 \\
                          \end{array}
                        \right)
                        -
                        \left(
                          \begin{array}{c}
                            n-t \\
                            3 \\
                          \end{array}
                        \right)
                        +p\left(
                           \begin{array}{c}
                             n \\
                             2 \\
                           \end{array}
                         \right)+n\left(
                                   \begin{array}{c}
                                     p \\
                                     2 \\
                                   \end{array}
                                 \right)+\left(
                                           \begin{array}{c}
                                             p \\
                                             3 \\
                                           \end{array}
                                         \right)
}{t(n-t)+np-\frac{p(n+2)}{4}
} .
\]
\end{Theorem}
\begin{proof}
Note that $I'$ is a squarefree monomial ideal generated by
monomials of degree 2 and 1. Let $k=\sdepth(I')$. The poset
$P_{I'}$ has the partition
$\mathcal{P}: P_{I'}=\bigcup_{i=1}^s[C_i,D_i]$, satisfying
$\sdepth(\mathcal{D(\mathcal{P})})=k$, where
$\mathcal{D(\mathcal{P})}$ is the Stanley decomposition of
$I'$ with respect to the partition $\mathcal{P}$. We may
 choose $\mathcal{P}$ such that $|D|=k$ whenever $C\neq D$
in the interval $[C,D]$.

For each interval $[C_i,D_i]$ in  $\mathcal{P}$ with
$|C_i|=2$ when in the corresponding monomial,
one variable belongs to $\{x_1,\ldots,x_t\}$ and one to
$\{x_{t+1},\ldots,x_n\}$ we have $|D_i|-|C_i|$ subsets of
cardinality 3 in this interval. Now for each interval
$[C_j,D_j]$ when $\mid C_j\mid=1$ we have at least
$\left(
 \begin{array}{c}
             k-1 \\
               2 \\
  \end{array}
 \right)
$ subsets of cardinality 3 in this interval. We have $p$
such intervals. So we have
$p\left(
     \begin{array}{c}
          k-1 \\
            2 \\
     \end{array}
\right)
$ subsets of cardinality 3.

 Now we consider those intervals
$[C_l,D_l]$ such that $|C_l|=2$ and the corresponding monomial
is of the form $x_l x_\lambda$, where
$x_l \in \{x_{n+1},\ldots,x_{n+p}\}$. Now either
$x_\lambda\in \{x_1,\ldots,x_n\}$ or
$x_\lambda\in \{x_{n+1},\ldots,x_{n+p}\}$. If
$x_\lambda\in \{x_1,\ldots,x_n\}$ then we have $np$
such intervals and each has at least $k-2$ subsets of
cardinality 3. If $x_\lambda\in\{x_{n+1},\ldots,x_{n+p}\}$
then we have
$\left(
   \begin{array}{c}
            p \\
            2 \\
   \end{array}
\right)
$ such intervals and each has at least $k-2$ subsets of
cardinality 3. Some subsets of cardinality 2 of the form
$C_l$ already appear in the intervals $[C_j,D_j]$ and such
subsets are $p(k-1)$ in number. Since the partition is
disjoint, we subtract this from total number of $C_l$'s,
so that we have at least
\[
\Big[\left(
    \begin{array}{c}
      n \\
      2 \\
    \end{array}
  \right)-\left(
            \begin{array}{c}
              t \\
              2 \\
            \end{array}
          \right)-\left(
                    \begin{array}{c}
                      n-t \\
                      2 \\
                    \end{array}
                  \right)\Big](k-2)+p\left(
                               \begin{array}{c}
                                 k-1 \\
                                 2 \\
                               \end{array}
                             \right)+\Big[np+\left(
                                           \begin{array}{c}
                                             p \\
                                             2 \\
                                           \end{array}
                                         \right)-p(k-1)
                             \Big](k-2)
\]
subsets of cardinality 3. This number is less than or equal
to the total number of subsets of cardinality 3.
So
\begin{equation} \label{ec1}
\begin{split}\Big[\left(
    \begin{array}{c}
      n \\
      2 \\
    \end{array}
  \right)-\left(
            \begin{array}{c}
              t \\
              2 \\
            \end{array}
          \right)-\left(
                    \begin{array}{c}
                      n-t \\
                      2 \\
                    \end{array}
                  \right)\Big](k-2)+
\Big[ np+\left(
                                           \begin{array}{c}
                                             p \\
                                             2 \\
                                           \end{array}
                                         \right)
-\frac{p(k-1)}{2}
                            \Big] (k-2)\\
\leq \left(
                          \begin{array}{c}
                            n \\
                            3 \\
                          \end{array}
                        \right)
                        -
                        \left(
                          \begin{array}{c}
                            t \\
                            3 \\
                          \end{array}
                        \right)
                        -
                        \left(
                          \begin{array}{c}
                            n-t \\
                            3 \\
                          \end{array}
                        \right)
                        +p\left(
                           \begin{array}{c}
                             n \\
                             2 \\
                           \end{array}
                         \right)+n\left(
                                   \begin{array}{c}
                                     p \\
                                     2 \\
                                   \end{array}
                                 \right)+\left(
                                           \begin{array}{c}
                                             p \\
                                             3 \\
                                           \end{array}
                                         \right).
\end{split}
\end{equation}
Now we know by~\cite[Theorem 2.11]{I} that
$k\leq \frac{n+2}{2}+p$. This implies
$-(k-1)\geq -\frac{n+2}{2}-p+1$. Using this in the left side
of inequality~\eqref{ec1}, one gets
\[
\Big[t(n-t)+np-\frac{p(n+2)}{4}
                             \Big](k-2)
\phantom{subsets of cardinality 3. This number is}
\]
\[
\leq \Big[\left(
    \begin{array}{c}
      n \\
      2 \\
    \end{array}
  \right)-\left(
            \begin{array}{c}
              t \\
              2 \\
            \end{array}
          \right)-\left(
                    \begin{array}{c}
                      n-t \\
                      2 \\
                    \end{array}
                  \right)\Big](k-2)
+\Big[ np+\left(
                                           \begin{array}{c}
                                             p \\
                                             2 \\
                                           \end{array}
                                         \right)
-\frac{p(k-1)}{2}
                             \Big] (k-2).
\]
Combining both inequalities we get the required result.
\end{proof}

\begin{Example}
Let us consider $I=(x_1,x_2,x_3)\cap (x_4,x_5,x_6)
\subseteq S=K[x_1,\ldots,x_6]$. By~\cite[Theorem 2.8]{I},
we have $\sdepth(I)\leq 4$. Let $I'=(I,x_7,x_8,x_9)\subseteq
S'=S[x_7,x_8,x_9]$ then by~\cite[Lemma 2.11]{I} we have
$\sdepth(I')\leq 7$. Now by our Theorem~\ref{te1} we have
$\sdepth(I')\leq 5$.
\end{Example}

We can further improve the upper bound if we impose some
additional condition on  $n$, $t$ and $p$.

Formula~\eqref{ec1} in the proof of Theorem~\ref{te1} is
equivalent to
\[
0\le 3 p k^2-3( 2n p+2 n t+ p^2-2 t^2+2 p)k
 \phantom{The last formula on page 1}
\]
\[ \phantom{The last formula}
+6 n p+6 n t+3 n^2 p-6 t^2-3 n t^2+3 n^2 t
+3 p^2+3 n p^2+2 p+p^3.
\]
Consider it as a quadratic polynomial in $k$ of
discriminant
\beq{ecd}
D:=(36 p t+36 t^2)n^2-36( t^2 p- t p^2+2 t^3)n
+12 p^2-36 p^2 t^2-3 p^4+36 t^4.
\eeq
Since this quadratic polynomial in $n$ has the
discriminant
\[
 \Delta:=432 t p^2 (t + p) (3 t^2 + 3 p t - 4 + p^2 )
\]
obviously positive, we have $D\ge 0$ for either
\[
 n\le t-\frac{p}{2}-p\sqrt{1+\frac{p^2-4}{3t(t+p)}}
\]
or
\beq{ecn}
  n\ge t-\frac{p}{2}+p\sqrt{1+\frac{p^2-4}{3t(t+p)}}.
\eeq
The former possibility is excluded by the fact that
$n>t$, so that, assuming the latter inequality, we
conclude that either
\[
k\le n+\frac{p}{2}+\frac{t(n-t)}{p}+1-\frac{\sqrt{D}}{6p}
\]
or
\beq{eck}
k\ge n+\frac{p}{2}+\frac{t(n-t)}{p}+1+\frac{\sqrt{D}}{6p}.
\eeq
The latter bound for $k$ does not hold (see Lemma~\ref{le1}
below). We have thus obtained the following result.

\begin{Theorem} \label{te2}
Keep the notation and hypotheses  of Theorem~\ref{te1}. If
\[
  n\ge t-\frac{p}{2}+p\sqrt{1+\frac{p^2-4}{3t(t+p)}}
\]
then
\beq{ect2}
\sdepth(I')\le n+\frac{p}{2}+\frac{t(n-t)}{p}+1-
\frac{\sqrt{D}}{6p},
\eeq
where
\[
D=(36 p t+36 t^2)n^2-36( t^2 p- t p^2+2 t^3)n
+12 p^2-36 p^2 t^2-3 p^4+36 t^4.
\]
\end{Theorem}

\begin{Lemma} \label{le1}
Conditions~\eqref{ecn} and~\eqref{eck} do not hold
simultaneously.
\end{Lemma}
\begin{proof}
Suppose that both inequalities~\eqref{ecn} and~\eqref{eck}
are satisfied. From the relation $k\le p+1+n/2$ known
from~\cite[Theorem 2.11]{I} it results, on the one hand, that
$p>n$ and, on the other hand, that $p(p-n)>2t(n-t)$, so that
\beq{ectrei}
p^2+2t^2>(p+2t)n.
\eeq
From ~\eqref{ecn} we obtain in particular
\[
 n>t+\frac{p}{2}.
\]
This and~\eqref{ectrei} give
\beq{ec5}
p>4t.
\eeq

We shall discuss two cases.

\textbf{Case $t\ge 2$.}
It is easily seen that the function $p\mapsto
\frac{p^2-4}{3t(t+p)}$ is increasing. Therefore
\[
 \frac{p^2-4}{3t(t+p)}>\frac{16 t^2-4}{15t^2}\ge 1.
\]
From~\eqref{ecn} it then follows $n>t+p(\sqrt{2}-0.5)
>t+0.91 p$. Using this in~\eqref{ectrei}, we get
$0.09 p>2.82 t$, whence $p>31 t$.  Then
\[
 \frac{p^2-4}{3t(t+p)}>\frac{961 t^2-4}{96t^2}\ge 10,
\]
so that $n>t+(\sqrt{11}-0.5)p>p$. This is a contradiction,
which shows that our assumption is false in this case.

\textbf{Case $t=1$.}
From $p>4$ we now get
\[
 \frac{p^2-4}{3(1+p)} \ge \frac{7}{6}
\]
and $n\ge 1+\left( \sqrt{\frac{13}{6}}-0.5\right) p>0.97p
+1$. Using this lower bound for $n$ in~~\eqref{ectrei},
we get $0.03 p>2.94 $, and therefore $p>98 >31 t$.  We
have seen that this  contradicts $p>n$.
\end{proof}

\noindent
\begin{Example}  \label{ex25}
For $n=7$, $t=3$, $p=5$, the latter theorem
gives $k\le 7$, while Theorem~\ref{te1} yields a slightly
weaker bound $k\le 8$. However, for  $n=66$, $t=2$, $p=3$
one gets $k\le 42$ by using Theorem~\ref{te2} and $k\le 41$
when applying Theorem~\ref{te1}.
\end{Example}
\bigskip
\begin{Corollary} \label{co1}
Let $I=Q\bigcap Q'$ be a monomial ideal in
$S=K[x_1,\ldots,x_n]$ where $Q$ and $Q'$ are monomial
primary ideals in $S$ such that $\sqrt{Q}=(x_1,\ldots,x_t)$
and $\sqrt{Q'}=(x_{r+1},\ldots,x_n)$ for some integers
$1\leq r\leq t<n$. Then
\begin{multline*}
\sdepth(I)\leq 2+\\ \frac{\Big(\!\!
                          \begin{array}{c}
                            n-t+r \\
                            3 \\
                          \end{array}
                        \!\!\Big)
                        -
                        \Big(\!\!
                          \begin{array}{c}
                            r \\
                            3 \\
                          \end{array}
                        \!\!\Big)
                        -
                        \Big(\!\!
                          \begin{array}{c}
                            n-t \\
                            3 \\
                          \end{array}
                        \!\!\Big)
                        +(t-r)\Big(\!\!
                           \begin{array}{c}
                             n-t+r \\
                             2 \\
                           \end{array}
                         \!\!\Big)+(n-t+r)\Big(\!\!
                                   \begin{array}{c}
                                     t-r \\
                                     2 \\
                                   \end{array}
                                 \!\!\Big)+\Big(\!\!
                                           \begin{array}{c}
                                             t-r \\
                                             3 \\
                                           \end{array}
                                         \!\!\Big)
}{r(n-t)+(n-t+r)(t-r)-\frac{(t-r)(n-t+r+2)}{4}
}.
\end{multline*}
\end{Corollary}
\begin{proof}
Note that $\sqrt{I}=(P'\cap S',x_{r+1},\ldots,x_t)$ where $S'=K[x_1,\ldots,x_r,x_{t+1},\ldots,x_n]$ and
$P'=(x_1,\ldots,x_r)\cap(x_{t+1},\ldots,x_n)\subset S'$.
Now we can apply Theorem~\ref{te1} and \cite[Theorem 2.1]{I}.
\end{proof}

\begin{Example}
Let $I=Q\cap Q'$ be a monomial ideal in $S=K[x_1,\ldots,x_8]$,
where $Q$ and $Q'$ are monomial primary ideals with
$\sqrt{Q}=(x_1,\ldots,x_6)$ and $\sqrt{Q'}=(x_5,\ldots,x_8)$.
Then by~\cite[Proposition 2.13]{I} we have $\sdepth{I}\leq 6$
and by our Corollary~\ref{co1} we have $\sdepth(I)\leq 5$.
\end{Example}

\section{Upper bounds for the Stanley depth of squarefree
Veronese ideal when some variables are added}  \label{sec3}

We denote by $I_{n,d}$ the squarefree Veronese ideal of degree
$d$ in the polynomial ring in $n$ variables over a field $K$.
Our first bound for the Stanley depth of such an ideal is given
by the next result.

\begin{Theorem} \label{te3}
 Let $K$ be a field and $n$, $p\ge 2$  integers.
Let $I_{n,2}$ be the squarefree Veronese ideal in
$S=K[x_1,\ldots,x_n]$ and $I'=(I_{n,2},x_{n+1},
\ldots,x_{n+p})\subseteq S'=S[x_{n+1},\ldots,x_{n+p}]$.
Then
\[\sdepth(I')\leq 2+\frac{\left(
                            \begin{array}{c}
                              n+p \\
                              3 \\
                            \end{array}
                          \right)
}{\left(
    \begin{array}{c}
      n \\
      2 \\
    \end{array}
  \right)+np-p- \frac{p}{2} \lfloor(
                                     \begin{array}{c}
                                       n \\
                                       3 \\
                                     \end{array}
                                   )/(
                                     \begin{array}{c}
                                       n \\
                                       2 \\
                                     \end{array}
                                   )
  \rfloor
}.
\]
\end{Theorem}
\begin{proof}
Note that $I'$ is a squarefree monomial ideal generated by
monomials of degree 2 and 1. Let $k=\sdepth(I')$. The poset
$P_{I'}$ has the partition
$\mathcal{P}: P_{I'}=\bigcup_{i=1}^s[C_i,D_i]$ satisfying $\sdepth(\mathcal{D(\mathcal{P})})=k$, where
$\mathcal{D(\mathcal{P})}$ is the Stanley decomposition of
$I'$ with respect to the partition $\mathcal{P}$. We may
choose $\mathcal{P}$ such that $|D|=k$ whenever $C\neq D$
in the interval $[C,D]$.

For each interval $[C_i,D_i]$ in
$\mathcal{P}$ with $|C_i|=2$, when in the corresponding
monomial both variables belong to $\{x_1,\ldots,x_n\}$ we
have at least $|D_i|-|C_i|$ subsets of cardinality 3 in this
interval. Now for each interval $[C_j,D_j]$, when $|C_j|=1$
we have at least
$\left(
     \begin{array}{c}
           k-1 \\
             2 \\
    \end{array}
\right)
$ subsets of cardinality 3 and we have $p$ such intervals.

Now we consider those intervals $[C_l,D_l]$ such that
$|C_l|=2$ and the corresponding monomial is of the form
$x_l x_\lambda$, where $x_l \in \{x_{n+1},\ldots,
x_{n+p}\}$. Now either $x_\lambda\in \{x_1,\ldots,x_n\}$
or $x_\lambda\in \{x_{n+1},\ldots,x_{n+p}\}$. If
$x_\lambda\in \{x_1,\ldots,x_n\}$, then we have $np$
such intervals and each of them has at least $k-2$ subsets
of cardinality 3. If
$x_\lambda\in\{x_{n+1},\ldots,x_{n+p}\}$ then we have
$\left(
   \begin{array}{c}
          p \\
          2 \\
   \end{array}
\right)
$ such intervals, each of which having at least $k-2$ subsets
of cardinality 3. Some subsets of cardinality 2 of the form
$C_l$ already appear in the interval when the interval starts
from a single variable, and there are $p(k-1)$ such subsets.
Since the partition is disjoint, we subtract this from the
total number of $C_l$'s, so that we have at least
\[\left(
    \begin{array}{c}
      n \\
      2 \\
    \end{array}
  \right)(k-2)+p\left(
                               \begin{array}{c}
                                 k-1 \\
                                 2 \\
                               \end{array}
                             \right)+\Big[ np+\left(
                                           \begin{array}{c}
                                             p \\
                                             2 \\
                                           \end{array}
                                         \right)-p(k-1)
                             \Big] (k-2)
\]
subsets of cardinality 3, and this number is less than or equal
to the total number of subsets of cardinality 3. So
\beq{ec6}
\left(
    \begin{array}{c}
      n \\
      2 \\
    \end{array}
  \right)(k-2)+p\left(
                               \begin{array}{c}
                                 k-1 \\
                                 2 \\
                               \end{array}
                             \right)+\Big[ np+\left(
                                           \begin{array}{c}
                                             p \\
                                             2 \\
                                           \end{array}
                                         \right)-p(k-1)
                             \Big] (k-2) \\
\leq \left(
                          \begin{array}{c}
                            n+p \\
                            3 \\
                          \end{array}\right) .
\eeq
Now
\begin{equation} \label{ec7}
\begin{split}
\left(
    \begin{array}{c}
      n \\
      2 \\
    \end{array}
  \right)(k-2)+p
    \left(
       \begin{array}{c}
              k-1 \\
                2 \\
       \end{array}
    \right)+\Big[ np+\left(
       \begin{array}{c}
                p \\
                2 \\
       \end{array}
                \right)-p(k-1)
                \Big] (k-2)\\
=\bigg[\left(
                               \begin{array}{c}
                                        n \\
                                        2 \\
                               \end{array}
                               \right)+np+\left(
                                          \begin{array}{c}
                                                 p \\
                                                 2 \\
                                          \end{array}
                                          \right)+\frac{p}{2}(1-k)
                             \bigg](k-2) .
\end{split}
\end{equation}
Since by~\cite[Theorem 1.2]{S} we know that
$\sdepth(I_{n,2})\leq \lfloor(
                      \begin{array}{c}
                               n \\
                               3 \\
                      \end{array}
                       )
/(
    \begin{array}{c}
      n \\
      2 \\
    \end{array}
  )
\rfloor + 2$, applying~\cite[Theorem 2.11]{I} we get
$k\leq \lfloor(
       \begin{array}{c}
                n \\
                3 \\
       \end{array}
       )
/(
    \begin{array}{c}
      n \\
      2 \\
    \end{array}
  )
\rfloor + 2+p$.

Putting
\[-k\geq -\lfloor(
          \begin{array}{c}
                 n \\
                 3 \\
          \end{array}
         )
/(
    \begin{array}{c}
      n \\
      2 \\
    \end{array}
  )
\rfloor - 2-p
\]
in~\eqref{ec7}, we get
\begin{equation*}
\begin{split}
\left(
    \begin{array}{c}
      n \\
      2 \\
    \end{array}
  \right)(k-2)+p\left(
                \begin{array}{c}
                     k-1 \\
                       2 \\
                \end{array}
                \right)+\Big[ np+\left(
                            \begin{array}{c}
                                    p \\
                                    2 \\
                            \end{array}
                            \right)-p(k-1)
                             \Big] (k-2)\\
\geq\bigg[\left(
          \begin{array}{c}
                  n \\
                  2 \\
          \end{array}
          \right)+np+\left(
                     \begin{array}{c}
                             p \\
                             2 \\
                     \end{array}
                     \right)+\frac{p}{2}\left( 1-\lfloor(
                                           \begin{array}{c}
                                                    n \\
                                                    3 \\
                                           \end{array}
                                           )
/(
    \begin{array}{c}
      n \\
      2 \\
    \end{array}
  )
\rfloor - 2-p\right)
                             \bigg](k-2).
\end{split}
\end{equation*}
The required result is obtained by  combining the above
inequality with~\eqref{ec6}.
\end{proof}

\begin{Example} \label{ex22}
Let $S=K[x_1,\ldots,x_5]$ and $I_{5,2}$ be the squarefree
Veronese ideal. Then by~\cite[Corollary 1.5]{MC}
or~\cite[Theorem 1.2]{S} we have $\sdepth(I_{5,2})=3$.

Now let $I'=(I_{5,2},x_6,x_7)$ be the monomial ideal in
$S'=S[x_6,x_7]$. By~\cite[Lemma 2.11]{I} we have
$\sdepth(I')\leq 5$, while  our Theorem~\ref{te3} yields
$\sdepth(I')\leq 4$.
\end{Example}

\begin{Example}
Let $S=K[x_1,\ldots,x_{11}]$ and $I_{11,2}$ be the squarefree
Veronese ideal. Then by~\cite[Theorem 1.2]{S} we have
$4\leq\sdepth(I_{11,2})\leq 5$.

Let $I'=(I_{11,2},x_{12},\ldots,x_{17})$ be the monomial ideal
in $S'=S[x_{12},\ldots,x_{17}]$, then by~\cite[Lemma 2.11]{I}
we have $\sdepth(I')\leq 11$ and by Theorem~\ref{te3}
$\sdepth(I')\leq 8$.
\end{Example}

If we impose some condition on $n$ and $p$ we can improve the
bound given in Theorem~\ref{te3}.

 The last expression given in the proof of Theorem~\ref{te3}
is equivalent to

\[
 0\le 3 pk^2-3(n^2- n+2 n p+ p^2+2 p)k
+ n^3+3 n^2 p+3 n^2+3 n p^2+6 n p-4 n+p^3+3 p^2+2 p.
\]

The quadratic in $k$ has discriminant
\[
E:=9n^4 +(24 p-18)n^3+(18 p^2-36 p+9)n^2-(18 p^2-12p)n
+12 p^2-3 p^4
\]
obviously positive for $n\ge p$. A simple computation
convince ourselves that the discriminant is actually positive
for $n\ge p-1$. Since, on the one hand, one has $E> 9(p-1)^4$
for $p\ge 2$, $n\ge \max \{ 2,p-1 \}$, and, on the other hand,
from~\cite[Theorem 1.2]{S} and~\cite[Theorem 2.11]{I}
it is  known that $k\le p+2+\lfloor (n-2)/3 \rfloor$,
we conclude that the next result holds.

\begin{Theorem} \label{te4}
Keep the notation and hypotheses from Theorem~\ref{te3}. Then
for $n\ge p-1$ one has
\[
k \le \frac{n(n-1)}{2p}+\frac{p}{2}+n+1-\frac{\sqrt{E}}{6p},
\]
where
\[
E=9n^4 +(24 p-18)n^3+(18 p^2-36 p+9)n^2-(18 p^2-12p)n
+12 p^2-3 p^4.
\]
 \end{Theorem}

\begin{Corollary}
Let $S'=K[x_1,\ldots,x_n,x_{n+1},\ldots,x_{n+p}]$ be a polynomial
ring and let $P_i=(x_1,\ldots,x_{i-1},x_{i+1},\ldots,x_n)$,
$i=1,\ldots, n$, be monomial prime ideals in $S'$. Denote
$Q_i=(P_i,x_{n+1},\ldots,x_{n+p})$. If
$\Ass(S'/I')=\{Q_1,\ldots,Q_n\}$, then
\[\sdepth(I)\leq 2+\frac{\left(
                            \begin{array}{c}
                              n+p \\
                              3 \\
                            \end{array}
                          \right)
}{\left(
    \begin{array}{c}
      n \\
      2 \\
    \end{array}
  \right)+np-p- \frac{p}{2} \lfloor(
                                     \begin{array}{c}
                                       n \\
                                       3 \\
                                     \end{array}
                                   )/(
                                     \begin{array}{c}
                                       n \\
                                       2 \\
                                     \end{array}
                                   )
  \rfloor
}.
\]
\end{Corollary}

\noindent
\begin{Example} For $n=11$, $p=6$, Theorem~\ref{te4}
gives $k\le 7$ instead of  $k\le 8$ cf. Theorem~\ref{te3}.
\end{Example}

\noindent
\begin{Example} For $n=5$, $p=2$, this result
gives $k\le 3$, while Theorem~\ref{te3} yields a slightly
weaker bound $k\le 4$. Therefore, in the situation described
in Example~\ref{ex22} one has
\[
 \sdepth_S(I)= \sdepth_{S'}(I').
\]
\end{Example}

We now prove that Stanley's conjecture is verified by ideals
of the type studied in this section.

\begin{Proposition} \label{pr1}
For positive integers $n$ and $d$,
let $I\subset S=K[x_1,\ldots,x_n]$ be the squarefree
Veronese ideal generated by all monomials of degree $d$ and
$I'=(I,x_{n+1},\ldots,x_{n+m})\subset S'=
S[x_{n+1},\ldots,x_{n+m}]$. Then
Stanley's conjecture holds for the  ideal $I'$.
\end{Proposition}
\begin{proof}
From $\depth_{S'}(S'/I')=\depth_S(S/I_{})$ it follows
$\depth_{S'} (I')=\depth_S (I)$. As a consequence of results
established in~\cite{HVZ} (or by
applying~\cite[Corollary~1.2]{MC}), Stanley's conjecture
holds for squarefree  Veronese ideals, so that
$\sdepth_S(I) \ge \depth_S(I)$. By~\cite[Lemma 2.1]{S1},
the sdepth does not decrease when passing from $I$ to $I'$.
Therefore, Stanley's conjecture holds for $I'$, too.
\end{proof}

\section{Comparison of bounds} \label{sec4}

First we compare the bounds provided in Theorems~\ref{te3}
and~\ref{te4}. The outcome of our study is the following.

\begin{Theorem} \label{te5}
 Let $K$ be a field and $n$, $p\ge 2$  integers.
Let $I_{n,2}$ be the squarefree Veronese ideal in
$S=K[x_1,\ldots,x_n]$ and $I'=(I_{n,2},x_{n+1},
\ldots,x_{n+p})\subseteq S'=S[x_{n+1},\ldots,x_{n+p}]$.
If $n\ge p-1$ then the bound for $\sdepth(I')$ given by
Theorem~\ref{te4} is smaller than that given by
Theorem~\ref{te3}.
\end{Theorem}
 Since
\[
 \left\lfloor \binom{n}{3} /  \binom{n}{2}\right\rfloor
= \left\lfloor (n-2)/3\right\rfloor,
\]
we shall distinguish the values of $n$ according to their
residues $\mod 3$.

\medskip

\textbf{Case $ n=3s+1$,  $s\ge 1$.}
The bound given in Theorem~\ref{te3} specialises to
\beq{ec11}
u_1:=\frac{27s^3+27s^2(p+2)+3s(3p^2+10p+5)+p^3+5p}{3s(9s+5p+3)+3p},
\eeq
while that given in Theorem~\ref{te4} becomes in this case
\beq{ec12}
l_1 :=\frac{27s^2+9(2p+1)s+3p^2+12p-\sqrt{dv_1}}{6p},
\eeq
with
\[
dv_1:=729s^4+(648p+486)s^3+(162p^2+324p+81)s^2+(54p^2+36p)s+
12p^2-3p^4.
\]

We want to know for what values of $s$ we have $l_1 \le u_1 $
for all $p\ge 2$,  or equivalently
\[
\left( 9s^2+(5p+3)s+p\right) \sqrt{dv_1} \ge r_1,
\]
where
\[
r_1:=243s^4+(243p+162)s^3+(63p^2+126p+27)s^2+(15p+27p^2-3p^3)s+
2p^2+3p^3-2p^4.
\]

The second derivative of the function function $r_1:[(p-2)/3,
+\infty ) \longrightarrow \RR$ being positive, the first
derivative is at least as large as
\[
r_1'(\frac{p-2}{3})=3(52p^3-153p^2+135p-36).
\]
Since the expression in the right side is positive for $p\ge 2$,
for these values $r_1$  is greater  than or equal to
\[
r_1(\frac{p-2}{3})=2(2p-3)(p-2)(2p-1)^2 ,
\]
which is nonnegative for  $p\ge 2$.
This analysis shows that the desired inequality is equivalent
to that obtained by squaring it, which, with some computer
assistance, is found to be
\[
  p ^2 (3 s + 1 + p) (3 s + p) (3 s - 1 + p)
  \left( s^2(15 p -18)+( 6 p^2 -3 p  -6) s - p^3  + 3p^2-2p\right)
\ge 0.
\]
This is true  if and only if
\[
f_1:=(15 p -18)s^2+( 6 p^2 -3 p  -6) s - p^3  + 3p^2-2p \ge 0
\quad \mbox{for} \  p\ge 2, s\ge \max \{2, (p-2)/3 \}.
\]
Since the discriminant
\[
 D_1=96p^4-288p^3+273p^2-108p+36
\]
is  positive for $p\ge 2$, $f_1 (s) \ge 0$ if and only if
\[
s\ge \frac{-( 6 p^2 -3p  -6)+\sqrt{D_1}}{6(5 p -6)} =:s_1.
\]

In terms of the number of variables $n$, this means that the
bound given in Theorem~\ref{te4} is better than that given in
Theorem~\ref{te3} for
\[
 n \ge n_1:=\frac{\sqrt{D_1}- 6 p^2 +13 p -6}{2(5 p -6)}.
\]

\medskip

\textbf{Case $ n=3s+2$,  $s\ge 0$.}
Now
\begin{align*}
u_2 & :=\frac{27s^3+27(p+3)s^2+(9p^2+48p+60)s+p^3+3p^2+%
14p+12}{27s^2+3(5p+9)s+6p+6},
\\
l_2 & :=\frac{27s^2+9(2p+3)s+3p^2+18p+6-\sqrt{dv2}}{6p},
\\
dv_2 & :=729s^4+(648p+1458)s^3+(162p^2+972p+1053)s^2
\phantom{The last formula}
\\
 & \phantom{The last formula i}
{}+(162p^2+468p+324)s-3p^4+48p^2+72p+36.
\end{align*}

Since $dv_2$ increases with $s$, its minimal value
in the range of interest is
\[
dv_2(\frac{p-3}{3})=3(p-2)(16p^3-40p^2+27p-6)\ge 0 .
\]
Therefore, $l_2\le u_2$ is equivalent to
\[
\left( 9s^2+(5p+9)s+2p+2\right) \sqrt{dv_2}\ge r_2,
\]
with
\[
r_2:=243s^4+243(p+2)s^3+(63p^2+351p+351)s^2
\phantom{The last formula}
\]
\[
\phantom{The ist foru}
{}+(-3p^3+57p^2+162p+108)s-2p^4+14p^2+24p+12.
\]

We further  find
\[
 r_2'\ge r_2'\left(\frac{p-3}{3}\right)=
3(52p^3-161p^2+141p-36) > 0  \quad  \mbox{for}\  p\ge 2,
\]
so that
\[
r_2\ge r_2 \left(\frac{p-3}{3}\right)=
2(8p^4-40p^3+61p^2-33p+6) > 0  \quad  \mbox{for}\ p\ge 3.
\]
As for $p=2$ the right side of the desired inequality
$l_2\le u_2$  is $3(3s+2)(27s^3+90s^2+85s+14)>0$, we may
square both sides of the  inequality under study and find
that for $p\ge 2$ it is  equivalent to
\[
p^3(p+2+3s)(p+1+3s)(p+3s)\left( 15s^2+(6p+15)s-p^2+3p+4\right)
\ge 0.
\]
This holds  precisely when
\[
f_2:=15s^2+(6p+ 15)s-p^2+3p+4 \ge 0.
\]
The discriminant being $96p^2-15>0$,
$f_2$ takes positive values for
\[
s\ge \frac{-3(2p+5)+\sqrt{96p^2-15}}{30}=:s_2.
\]

Thus we conclude that $l_2\le u_2$ holds for $n\equiv 2 \pmod 3$
and
\[
n\ge n_2:= \frac{\sqrt{96p^2-15}-6p+5}{10}.
\]

\medskip

\textbf{Case $ n=3s$,  $s\ge 1$.}
We study the inequality $l_3\le u_3$, with
\begin{align*}
l_3 & :=\frac{27s^2+9(2p-1)s+3p^2+6p-\sqrt{dv_3}}{6p},
\\
u_3 & :=\frac{27s^3+27(p+1)s^2+(9p^2+12p-12)s+p^3-3p^2-
4p}{27s^2+3(5p-3)s-3p},
\end{align*}
and
\[
dv_3:= 729s^4+(648p-486)s^3+(162p^2-324p+81)s^2+
(-54p^2+36p)s+12p^2-3p^4.
\]
With arguments similar to those given in the previous cases
one finds
\[
dv_3\ge dv_3(\frac{p-1}{3})=3(p-2)(16p^3-40p^2+27p-6) \ge 0
\quad \mbox{for} \  p\ge 2.
\]
Therefore, $l_3\le u_3$ is equivalent to
\[
 \left( 9s^2+(5p-3)s-p \right) \sqrt{dv_3} \ge r_3,
\]
where
\[
r_3:=243s^4+(243p-162)s^3+(63p^2-126p+27)s^2
+(-3p^3-21p^2+15p)s+2p^2+3p^3-2p^4.
\]
After we  check that $r_3$ is positive in the range $p\ge 2$,
$s\ge \max \{ 1, (p-1)/3 \}$, we may square the last inequality
and find that it is equivalent to
\[4p^2(3s+p-2)(3s+p)(3s+p-1)
\left( s^2(15p-9)+(6p^2-9p+3)s-p^3+p\right) \ge 0.
\]
Since the quadratic polynomial in $s$
\[f_3=(15p-9)s^2+(6p^2-9p+3)s-p^3+p
\]
has discriminant
\[
D_3:=3(p-1)(32p^3-16p^2+3p-3) >0,
\]
we have $f_3 (s) \ge 0$ if and only if
\[
s\ge \frac{-6p^2+9p-3+\sqrt{D_3}}{6(5p-3)}=:s_3.
\]

The conclusion is that, for $n\equiv 3 \pmod 3$, the bound
provided in Theorem~\ref{te4} is tighter than that given
in Theorem~\ref{te3} if and only if
\[
n\ge \frac{-6p^2+9p-3+\sqrt{D_3}}{2(5p-3)}=:n_3.
\]

\medskip

It remains to compare  $n_1$, $n_2$, $n_3$ and $p-1$.

\begin{Lemma}
One has  $n_1=1$, $n_2\simeq 1.22$, $n_3\simeq 1.07$
for $p=2$, and  $p-1\ge n_2 > n_3 > n_1$ for $p\ge 3$.
\end{Lemma}
\begin{proof}
Assume $p\ge 3$. The inequality $n_2>n_3$ is successively
equivalent to
\[
(5p-3)\sqrt{96p^2-15}>2p+5\sqrt{D_3} ,
\]
\[
36p^3-47p^2+45p-18 >p \sqrt{D_3},
\]
and
\[
(5p-3)^2(12p^4-18p^3+28p^2-15p+9)>0,
\]
which is obviously true.

The inequality $n_1<n_3$ is rewritten
\[2p^2+(5p-3) \sqrt{D_1} < (5p-6)\sqrt{D_3}.
\]
Squaring this, one finds after some easy computations
\[
 (5p-3)p \sqrt{D_1} < (5p-3)(36p^3-119p^2+150p-72).
\]
After simplification and squaring one gets
\[
 1200p^6-8424p^5+24904p^4-40866p^3+39627p^2-21600p+5184>0,
\]
which is readily checked to be true for $p\ge 3$.

Finally, $n_2<p-1$ is put into the equivalent form
\[
 \sqrt{96p^2-15}<16p-5p,
\]
which holds because the left side is less than $10p$, while
the right side is at least $11p$ for $p\ge 3$.
\end{proof}

Now the proof of Theorem~\ref{te5} is complete.

\medskip

The bounds for the class of ideals studied in Section~\ref{sec2}
can be compared by analogue reasoning. The details of the
analysis are, however, much more involved. As seen by
Example~\ref{ex25}, none of Theorems~\ref{te1} and~\ref{te2}
is uniformly better  than the other. Our final
result specifies conditions under which Theorem~\ref{te2}
yields a tighter bound than that given in Theorem~\ref{te1}.

\begin{Theorem} \label{te6}
Let $I=(x_1,\ldots,x_t)\bigcap(x_{t+1},\ldots,x_n)$ be a
monomial ideal in \linebreak
$S =K[x_1,\ldots,x_n]$, where $1\leq t< n$,
and let $I'=(I,x_{n+1},\ldots,x_{n+p}) \subset S'=S[x_{n+1},
\ldots, x_{n+p}]$, where $p\geq 2$.  Suppose that it holds
\[
  n\ge n_0:=t-\frac{p}{2}+p\sqrt{1+\frac{p^2-4}{3t(t+p)}}.
\]

Then the bound for $\sdepth(I')$ given in~\eqref{ect2} is
tighter than that given in~\eqref{ect1} if and only if
\[
 0\le 3n^2+6np-4p^2+4 \quad \text{and} \quad
\max \{ 1, t_l \} \le t \le \min \{ n-1, t_u\},
\]
where
\[
 t_l:=\frac{6n-\sqrt{6(3n^2+6np-4p^2+4)}}{12},
\quad
 t_u:=\frac{6n+\sqrt{6(3n^2+6np-4p^2+4)}}{12}.
\]
\end{Theorem}
\begin{proof}
For $n=2$ one has $t=1$ and therefore (by hypothesis
$n\ge n_0$) $p=2$, so that the
bounds given in Theorems~\ref{te1} and~\ref{te2} are
$10/3$ and $3$, respectively. From now on we shall assume
$n\ge 3$.

With notation
\begin{align*}
 L & :=n+\frac{p}{2}+\frac{t(n-t)}{p}+1-
\frac{\sqrt{D}}{6p},
\\
 U &  :=2+\frac{\left(\!\!
                          \begin{array}{c}
                            n \\
                            3 \\
                          \end{array}
                        \!\!\right)
                        -
                        \left(\!\!
                          \begin{array}{c}
                            t \\
                            3 \\
                          \end{array}
                        \!\!\right)
                        -
                        \left(\!\!
                          \begin{array}{c}
                            n-t \\
                            3 \\
                          \end{array}
                        \!\!\right)
                        +p\left(\!\!
                           \begin{array}{c}
                             n \\
                             2 \\
                           \end{array}
                         \!\!\right)+n\left(\!\!
                                   \begin{array}{c}
                                     p \\
                                     2 \\
                                   \end{array}
                                \!\! \right)+\left(\!\!
                                           \begin{array}{c}
                                             p \\
                                             3 \\
                                           \end{array}
                                        \!\! \right)
}{t(n-t)+np-\frac{p(n+2)}{4}
} ,
\end{align*}
we have to find when does it hold $L\le U$. Routine computations
bring this inequality to the equivalent form
\beq{ec41}
f_4\le g_4 \sqrt{D},
\eeq
with
\begin{align*}
f_4 & :=(6 p^2+30 t p+24 t^2)n^2+(-3 p^3+12 p^2 t-6 p^2-12 t p-
30 t^2 p-48 t^3)n  \\ 
& \phantom{find when doe does itdoes i}
{}-4 p^4+6 p^3+4 p^2-12 p^2 t^2+12 t^2 p+24 t^4,
\\
g_4 & :=4nt+3np-4t^2-2p,
\\
 D & :=(36 p t+36 t^2)n^2-36( t^2 p- t p^2+2 t^3)n
+12 p^2-36 p^2 t^2-3 p^4+36 t^4.
\end{align*}
The discriminant of $f_4$, which is found to be
\[
 df_4:=3p^2\left( 35p^4+(136t-36)p^3+(332t^2-264t-20)p^2
\right. \phantom{find when does it}
\]
\[
\left. \phantom{fnd whe}{}+
(336t^3-264t^2-112t)p+108t^4-48t^3-80t^2\right),
\]
is  positive in our hypothesis (for $t\ge 2$ the coefficients
of powers of $p$ are obviously positive, and a direct verification
leads to the same conclusion if $t=1$). Therefore, $f_4$ takes
nonnegative values for either
\[
n\le n_l:= \frac{3 p^3-12 p^2 t+6 p^2+12 t p+
30 t^2 p+48 t^3-\sqrt{df_4}}{2(6 p^2+30 t p+24 t^2)}
\]
or
\[
n\ge n_u:= \frac{3 p^3-12 p^2 t+6 p^2+12 t p+
30 t^2 p+48 t^3+\sqrt{df_4}}{2(6 p^2+30 t p+24 t^2)}.
\]
As will shall prove in Lemma~\ref{le3} below, one has
$n_u\le n_0$. Therefore, the hypothesis of Theorem~\ref{te6}
ensures that Eq.~\eqref{ec41} is equivalent to
\[
h:= g_4^2 D-f_4^2\ge 0..
\]
With some computer assistance, we find
\[
 h=4p^3 h_1h_2,
\]
where the quadratic polynomials in $t$
\begin{align*}
h_1 & :=3(n-2)t^2-(3n^2-6n)t+3n^2p+3np^2-6np+2p-3p^2+p^3, \\
h_2 & :=24t^2-24nt+3n^2-6np+4p^2-4
\end{align*}
have discriminant
\[
\Delta _1:=3(n-2)(3n^3-6n^2+12n^2p+12np^2-24np+4p^3-12p^2+8p)
\]
and respectively
\[
 \Delta _2:=288n^2+576np-384p^2+384.
\]
Each discriminant is increasing with $n$, so that
\[
 \Delta _1 \ge \Delta _1 (3)=3(4p^3+24p^2+44p+27) > 0.
\]
Hence, $h_1$ always has the real roots
\[
t_1:=\frac{3n^2-6n-\sqrt{\Delta_1}}{6(n-2)},
\quad
t_2:=\frac{3n^2-6n+\sqrt{\Delta_1}}{6(n-2)},
\]
while $h_2$  has the real roots
\[
 t_3:=\frac{6n-\sqrt{6(3n^2+6np-4p^2+4)}}{12},
\quad
 t_4:=\frac{6n+\sqrt{6(3n^2+6np-4p^2+4)}}{12}
\]
provided that
\[
 0\le 3n^2+6np-4p^2+4.
\]

In Lemma~\ref{le4} below we show that $t_1<1$ and  $n-1<t_2$.
Therefore, $h_1$ is negative for all admissible values of $t$,
whence $h\ge 0$ is equivalent to  $h_2 \le 0$. The latter
inequality is valid precisely when the conclusion of
Theorem~\ref{te6} holds.
\end{proof}

The proof of Theorem~\ref{te6} is complete as soon as we prove
the next lemmas.
\begin{Lemma} \label{le3}
One has $n_u\le n_0$.
\end{Lemma}
\begin{proof}
The desired inequality is successively equivalent to
\[
(60tp+12p^2+48t^2)\sqrt{1+\frac{p^2-4}{3t(t+p)}}+
6t^2-6tp-9p^2-12t-6p\ge \sqrt{df_4}
\]
and
\[
 16p^7+116p^6+(776t^2-128)p^5+4t(763t^2+66t-256)p^4
+(6837t^4+876t^3-5068t^2+256)p^3
\]
\[
{}+t(8517t^4+1884t^3-11708t^2-1056t+2240)p^2
+8t^2(711t^4+225t^3-1578t^2-204t+712)p
\]
\[
\phantom{the last inequality becom}
{}+4t^3(135t^2+48t-244)(3t^2-4)\ge 0.
\]

For $t=1$, the last inequality becomes
\[
(p+2)(16p^6+84p^5+480p^4+1332p^3+237p^2-597p+122)\ge 0,
\]
for $t=2$
\[
8(p+4)(2p^3+9p^2+210p+392)(p+2)^3 \ge 0,
\]
while for $t\ge 3$ all powers of $p$ have positive coefficients.
\end{proof}

\begin{Lemma} \label{le4}
One has $t_1< 1$, $ n-1 < t_2 $. 
\end{Lemma}
\begin{proof}
The inequality $t_1< 1$ is readily brought to the equivalent
form
\[
0< (p+1) \bigl( 3n^2+3(p-3)n+p^2-4p+6\bigr),
\]
which is obviously true for $n\ge \max \{ 2, p-1\}$. Since
$t_1+t_2=n$, we also have $ n-1 < t_2 $.
\end{proof}

\end{document}